\documentclass[11pt]{amsart}
\usepackage{amscd, amssymb, pgf, tikz, hyperref}

\newcommand{\field}[1]{\mathbb{#1}}
\newcommand{\CC}{\field{C}}

\newcommand{\MM}{\field{M}}
\newcommand{\TT}{\field{T}}
\newcommand{\ZZ}{\field{Z}}

\newcommand{\Gg}{\mathcal{G}}
\newcommand{\Hh}{\mathcal{H}}

\newcommand{\Ll}{\mathcal{L}}
\newcommand{\Mm}{\mathcal{M}}

\newcommand{\Pp}{\mathcal{P}}

\newcommand{\Xx}{\mathcal{X}}

\newcommand{\la}{\langle}
\newcommand{\ra}{\rangle}

\newcommand{\ds}{\displaystyle}

\newtheorem{thm}{Theorem}[section]

\newtheorem{prop}[thm]{Proposition}

\theoremstyle{definition}
\newtheorem{dfn}[thm]{Definition}

\theoremstyle{remark}
\newtheorem{rmk}[thm]{Remark}

\newtheorem{example}[thm]{Example}

\newtheorem*{examples*}{Examples}

\numberwithin{equation}{subsection}
\numberwithin{equation}{subsection}

\title{ The Koopman representation  for self-similar groupoid actions}

\author{Valentin Deaconu}
\address{Valentin Deaconu \\ Department of Mathematics (0084)\\ University
of Nevada\\ Reno NV 89557-0084\\ USA} \email{vdeaconu@unr.edu}



\begin{document}
\begin{abstract}
We introduce the $C^*$-algebra $C^*(\kappa)$ generated by the Koopman representation $\kappa$ of an \' etale groupoid $G$ acting on a measure space $(X,\mu)$. We prove that for a level transitive self-similar action $(G,E)$ with $E$ finite and $|uE^1|$ constant, there is an invariant measure $\nu$ on $X=E^\infty$ and  that $C^*(\kappa)$  is residually finite-dimensional with a normalized self-similar trace.
We also discus $p$-fold similarities of Hilbert spaces in connection to representations of the graph algebra $C^*(E)$ and self-similar representations of $G$ in connection to the Cuntz-Pimsner algebra $C^*(G,E)$.
 
\end{abstract}
\maketitle
\section{introduction}

\bigskip
The concept of a group action on a space was generalized to a groupoid action and it has applications to dynamical systems, representation theory and operator algebras. If groups can roughly be described  as the set of symmetries of certain objects, then groupoids can be thought as the set of symmetries of fibered objects.

Self-similar group actions are defined by using subgroups of the automorphism group of a rooted tree which is  viewed as the path space of a graph with one vertex and $n\ge 2$ edges. Given a finite directed graph $E$ with no sources, the corresponding path space is fibered over $E^0$ and it gives rise to a family of trees (or forest) $T_E$.  It is natural to consider self-similar actions of subgroupoids $G$ of PIso$(T_E)$, the discrete groupoid of partial isomorphsims of $T_E$. The most interesting case is when $|uE^1|$ is constant for $u\in E^0$, since then all trees in $T_E$ are isomorphic. 

We begin  by recalling the definition of a groupoid action on a space $X$ and of the Koopman representation $\kappa$ associated to an invariant measure $\mu$ on $X$. Self-similar groupoid actions $(G,E)$ were introduced in \cite{LRRW} and in the presence of an invariant measure on $X=E^\infty$, we discuss the $C^*$-algebra $C^*(\kappa)$ associated to the induced Koopman representation. We prove that $C^*(\kappa)$ is residually finite-dimensional (RFD) and it has a self-similar trace $\tau_0$.

We then define the wreath recursion for $(G,E)$ and the matrix recursion for $C_c(G)$. We recall the concept of $p$-fold similarities of Hilbert spaces and relate it to representations of the graph algebra $C^*(E)$. We also  define self-similar representations of $G$ and relate these  to representations of the Cuntz-Pimsner algebra $C^*(G,E)$. We illustrate with several examples.

\bigskip

\section{Groupoid actions and the Koopman representation}

\bigskip

A groupoid $G$ is a  small category with inverses. We will use $d$ and $t$ for the domain and target maps $d,t:G \to G^{(0)}$ to distinguish them from the source and range maps $s,r:E^1\to E^0$ on directed graphs. An \'etale groupoid is a topological groupoid where the target map $t$ (and necessarily the
domain map  $d$) is a local homeomorphism. The unit space $G^{(0)}$ of an \'etale groupoid is always an open subset of $G$. The set of composable pairs is denoted $G^{(2)}$ and the set of $g\in G$ with $d(g)=u, t(g)=v$ is denoted $G_u^v$. Two units $x,y \in G^{(0)}$ belong to the same $G$-orbit if there exists $g \in G$ such that $d(g) = x$ and $t(g) = y$. When every $G$-orbit is dense in $G^{(0)}$, then the groupoid $G$ is called minimal. 

The isotropy group of a unit $x\in G^{(0)}$ is the group \[G_x^x :=\{g\in G\; | \; d(g)=t(g)=x\},\] and the isotropy bundle is
\[G' :=\{g\in G\; | \; d(g)=t(g)\}= \bigcup_{x\in G^{(0)}} G_x^x.\]
A groupoid $G$ is said to be principal if all isotropy groups are trivial, or equivalently, $G' = G^{(0)}$. 

\begin{dfn} Let $G$ be an \' etale groupoid. A bisection is an open subset $U\subseteq  G$ such that $d$ and $t$ are both injective when restricted to $U$.
\end{dfn}

 We now recall the definition of a groupoid action on a space given in \cite{MRW1}:

\begin{dfn}\label{sa}  A topological groupoid $G$ is said to act
(on the left) on a locally compact space $X$, if there are given 
a continuous open surjection $\omega : X \rightarrow G^{(0)}$,  called the anchor or moment map,
and a continuous map
\[G\ast X \rightarrow X, \quad\text{write}\quad (g , x)\mapsto
g \cdot x,\]
where
\[G \ast X = \{(g , x)\in G \times X \mid d(g) = \omega (x)\},\]
that satisfy

\medskip

i) $\omega (g \cdot x) =t (g)$ for all $(g , x) \in G \ast X,$

\medskip

ii) $(g _2, x) \in G \ast X,\,\, (g_1, g_2)
\in G ^{(2)}$ implies $(g _1g _2, x),
(g _1, g _2\cdot x) \in G * X$ and
\[g _1\cdot(g _2\cdot x) = (g _1g _2)
\cdot x,\]

\medskip

iii) $\omega (x)\cdot x = x$ for all $x\in X$. 
\end{dfn}
We should mention that recently, some authors don't necessarily assume that $\omega : X\to G^{(0)}$ is   open or surjective (see  \cite{W} for example).

The action of $G$ on $X$  is called transitive if given $x,y\in X$, there is $g\in G$ with $g\cdot x=y$ and is free if $g\cdot x=x$ for some $x$ implies $g=\omega(x)$. 

The set of fixed points is defined as \[X^G=\{x\in X: g\cdot x=x\;\text{for all}\; g\in G_{\omega(x)}^{\omega(x)}\}.\] 
If $G$ has trivial isotropy, then $X^G=X$.

For $x\in X$, its  stabilizer group is
\[G(x)=\{g\in G: g\cdot x=x\},\] which is a subgroup of  $G_u^u$ where $u=\omega(x)$. The orbit of $x\in X$ is \[G\cdot x=\{g\cdot x: g\in G, \; d(g)=\omega(x)\}.\] The set of orbits is denoted by $G\backslash X$ and has the quotient topology. The action of $G$ on $X$ is called minimal if every orbit $G\cdot x$ is dense in $X$.

\begin{example} A groupoid $G$ with open domain and target maps acts on its unit space $G^{(0)}$ by $g\cdot d(g)=t(g)$. In this case, $\omega=id$. The groupoid is called transitive if this action is transitive. Notice that $g\cdot u=u$ for all $g\in G_u^u$, in particular $(G^{(0)})^G=G^{(0)}$.  A transitive groupoid with discrete unit space is of the form $G^{(0)}\times K\times G^{(0)}$ where $K$ is the isotropy group. 
\end{example}

\begin{rmk}The fibered product $G \ast X = \{(g , x)\in G \times X \mid d(g) = \omega (x)\}$ has a natural
structure of  groupoid, called the semi-direct product or
action groupoid and is denoted  by $G \ltimes X$, where
\[(G \ltimes X)^{(2)} = \{ ((g_1, x_1),(g _2, 
x_2)) \mid \,\,x_1 = g _2\cdot x_2\},\]
with operations
\[(g _1, g_2\cdot x_2)(g _2,  x_2) = 
(g _1g _2, x_2), \;\;
(g,x)^{-1} = (g ^{-1}, g \cdot x).\]
The domain and target maps of $G \ltimes X$ are
 \[d(g ,x) = (d(g),x)=
(\omega( x), x),\quad t(g ,x) =  (t(g), g\cdot x)=(\omega(g\cdot  x),g\cdot x),\]
and the unit space $(G\ltimes X)^{(0)}$ may be identified with $X$ via the map \[i:X\to G\ltimes X,\;\; i(x)= (\omega(x),x).\] The projection map\[\pi:G\ltimes X\to G, \;\; \pi(g,x)=g\] is a covering of groupoids. 
\end{rmk}

Assume that the topological groupoid $G$ acts on the space $X$ via $\omega:X\to G^{(0)}$. A Borel measure $\mu$ on $X$ is $G$-invariant if $\mu(g^{-1}\cdot A)=\mu(A)$ for any Borel set $A\subseteq \omega^{-1}(t(g))$. 

\begin{dfn} For a $G$-invariant measure $\mu$ on $X$, the Hilbert space $L^2(X,\mu)$ is fibered over $G^{(0)}$ with fibers $L^2(\omega^{-1}(u),\mu)$ for $u\in G^{(0)}$ and we  define the Koopman representation $\kappa$ of $G$ on $L^2(X,\mu)$ by
\[\kappa_g:L^2(\omega^{-1}(d(g)),\mu)\to L^2(\omega^{-1}(t(g)),\mu),\;\; \kappa_g(f)(x)=f(g^{-1}\cdot x).\]
\end{dfn}
It can be verified that each $\kappa_g$ is a partial isometry in $\Ll(L^2(X,\mu))$ and  if $G$ is \'etale (our main case of interest), the representation $\kappa:G\to\Ll(L^2(X,\mu))$  can be extended to the  algebra $C_c(G)$ of compactly supported continuous functions.  We denote by $C^*(\kappa)$  the closure of $\kappa(C_c(G))$ in the operator norm of $\Ll(L^2(X,\mu))$.

For $G$ a discrete groupoid acting on  $X$, we can also define a family of representations $\{\lambda_x\}_{ x\in X}$ on the orbit Hilbert spaces $\ell^2(G\cdot x)$ such that \[\lambda_x(g)\xi(y)=\xi(g^{-1}\cdot y)\] for $\xi\in \ell^2(G\cdot x)$. Note that $y=h\cdot x$ where $d(h)=\omega(x)$ and $g^{-1}\cdot y=(g^{-1}h)\cdot x\in G\cdot x$ since $d(g^{-1})=\omega(h\cdot x)=t(h)$. 

If $\mu$ is a measure on $X$, by taking the integral $\ds\Hh=\int_X\ell^2(G\cdot x)d\mu(x)$, we get a representation $\lambda:G\to \Ll(\Hh)$ which can be extended to a representation of $C_c(G)$. The relationship between the norm closure $C^*(\lambda)$ of $\lambda(C_c(G))$ in $\Ll(\Hh)$, the above $C^*$-algebra $C^*(\kappa)$ obtained from the Koopman representation and the reduced $C^*$-algebra $C^*_r(G)$ remains to be explored. Recall that for an \'etale groupoid, $C^*_r(G)$ is defined as the completion of $C_c(G)$ in the norm $\|a\|=\sup_{u\in G^{(0)}}\|\pi_\lambda^u(a)\|$ where
\[\pi_\lambda^u(a)\xi(g)=\sum_{r(h)=r(g)}a(h)\xi(h^{-1}g),\]
for $a\in C_c(G)$ and $\xi\in \ell^2(s^{-1}(u))$.

There are   other interesting $C^*$-algebras associated to a groupoid action, like $C^*(G\ltimes X)\cong C_0(X)\rtimes G$. In the case of a self-similar groupoid action $(G,E)$,   the Cuntz-Pimsner algebra $C^*(G,E)$ is constructed from a $C^*(G)$-correspondence or from an \'etale groupoid $\Gg(G,E)$, see \cite{D} for details.
\bigskip

\section{Self-similar actions}

\bigskip

Let $E = (E^0, E^1, r, s)$ be a finite directed graph  with no sources. For $k \ge 2$, define  the set of paths of length $k$ in $E$ as
\[E^k = \{e_1e_2\cdots e_k : e_i \in E^1,\; r(e_{i+1}) = s(e_i)\}.\]
The maps $r,s$ are naturally extended to $E^k$ by taking
\[r(e_1e_2\cdots e_k)=r(e_1),\;\; s(e_1e_2\cdots e_k)=s(e_k).\]
We denote by $E^* :=\bigcup_{ k\ge 0} E^k$  the space of finite paths (including vertices) and  by $E^\infty$ the infinite path space of $E$ with the usual topology given by the cylinder sets $Z(\alpha)=\{\alpha\xi:\xi\in E^\infty\}$ for $\alpha\in E^*$. 
Both spaces $E^*$ and $E^\infty$ are fibered over $E^0$ using the range map.
We know that $E^\infty$ is a Cantor space precisely when $E$ satisfies condition (L), i.e. every cycle has an entrance.
 
 We can visualize the set  $E^*$ as indexing the vertices of a union of rooted trees  or forest $T_E$   given by $T_E^0 =E^*$ and with edges \[T_E^1 =\{(\alpha, \alpha e) :\alpha\in E^*, e\in E^1\; \text{and}\; s(\alpha)=r(e)\}.\]
 
 \begin{example}\label{ex}
 
 For the graph
 
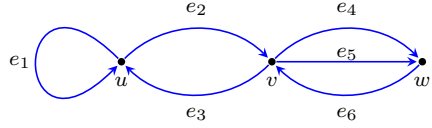
\[\begin{tikzpicture}[shorten >=0.4pt,>=stealth, semithick]
\renewcommand{\ss}{\scriptstyle}
\node[inner sep=1.0pt, circle, fill=black]  (u) at (-2,0) {};
\node[below] at (u.south)  {$\ss u$};
\node[inner sep=1.0pt, circle, fill=black]  (v) at (0,0) {};
\node[below] at (v.south)  {$\ss v$};
\node[inner sep=1.0pt, circle, fill=black]  (w) at (2,0) {};
\node[below] at (w.south)  {$\ss w$};

\draw[->, blue] (u) to [out=45, in=135]  (v);
\node at (-1,0.7){$\ss e_2$};
\draw[->, blue] (v) to [out=-135, in=-45]  (u);
\node at (-1,-0.7) {$\ss e_3$};
\draw[->, blue] (v) to [out=45, in=135]  (w);
\node at (1,0.7){$\ss e_4$};
\draw[->, blue] (v) to (w);
\node at (1,0.1){$\ss e_5$};
\draw[->, blue] (w) to [out=-135, in=-45]  (v);
\node at (1,-0.7) {$\ss e_6$};

\draw[->, blue] (u) .. controls (-3.5,1.5) and (-3.5, -1.5) .. (u);
\node at (-3.35,0) {$\ss e_1$};

\end{tikzpicture}
\]

the forest $T_E$ looks like

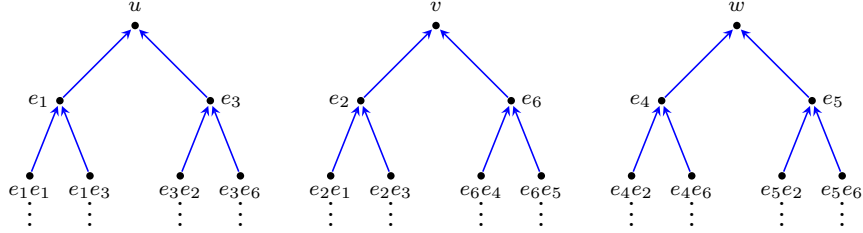
\[
\begin{tikzpicture}[shorten >=0.4pt,>=stealth, semithick]
\renewcommand{\ss}{\scriptstyle}
\node[inner sep=1.0pt, circle, fill=black]  (u) at (-4,4) {};
\node[above] at (u.north)  {$\ss u$};
\node[inner sep=1.0pt, circle, fill=black]  (v) at (0,4) {};
\node[above] at (v.north)  {$\ss v$};
\node[inner sep=1.0pt, circle, fill=black]  (w) at (4,4) {};
\node[above] at (w.north)  {$\ss w$};

\node[inner sep=1.0pt, circle, fill=black] (e1) at (-5,3){};
\node[left] at (-5,3)  {$\ss e_1$};
\node[inner sep=1.0pt, circle, fill=black] (e3) at (-3,3){};
\node[right] at (-3,3)  {$\ss e_3$};
\draw[->, blue] (e1) to  (u);
\draw[-> , blue] (e3) to  (u);
\node[inner sep=1.0pt, circle, fill=black] (e11) at (-5.4,2){};
\node[inner sep=1.0pt, circle, fill=black] (e13) at (-4.6,2){};
\node[below] at (-5.4,2)  {$\ss e_1e_1$};
\node[below] at (-5.4,2) {$\vdots$};
\draw[-> , blue] (e11) to   (e1);
\node[below] at (-4.6,2)  {$\ss e_1e_3$};
\node[below] at (-4.6,2) {$\vdots$};
\draw[-> , blue] (e13) to  (e1);
\node[inner sep=1.0pt, circle, fill=black] (e32) at (-3.4,2){};
\node[inner sep=1.0pt, circle, fill=black] (e36) at (-2.6,2){};
\node[below] at (-3.4,2) {$\ss e_3e_2$};
\node[below] at (-3.4,2) {$\vdots$};
\node[below] at (-2.6,2) {$\ss e_3e_6$};
\node[below] at (-2.6,2) {$\vdots$};
\draw[-> , blue] (e32) to   (e3);
\draw[-> , blue] (e36) to   (e3);

\node[inner sep=1.0pt, circle, fill=black] (e2) at (-1,3){};
\node[left] at (-1,3)  {$\ss e_2$};

\node[inner sep=1.0pt, circle, fill=black] (e6) at (1,3){};
\node[right] at (1,3)  {$\ss e_6$};
\draw[-> , blue] (e2) to  (v);
\draw[-> , blue] (e6) to  (v);
\node[inner sep=1.0pt, circle, fill=black] (e21) at (-1.4,2){};
\node[inner sep=1.0pt, circle, fill=black] (e23) at (-0.6,2){};

\draw[-> , blue] (e21) to   (e2);

\draw[-> , blue] (e23) to  (e2);
\node[inner sep=1.0pt, circle, fill=black] (e64) at (0.6,2){};
\node[inner sep=1.0pt, circle, fill=black] (e65) at (1.4,2){};
\draw[-> , blue] (e64) to   (e6);
\draw[-> , blue] (e65) to   (e6);
\node[below] at (-1.4,2)  {$\ss e_2e_1$};
\node[below] at (-1.4,2) {$\vdots$};
\node[below] at (-0.6,2)  {$\ss e_2e_3$};
\node[below] at (-0.6,2) {$\vdots$};
\node[below] at (0.6,2) {$\ss e_6e_4$};
\node[below] at (0.6,2) {$\vdots$};
\node[below] at (1.4,2) {$\ss e_6e_5$};
\node[below] at (1.4,2) {$\vdots$};

\node[inner sep=1.0pt, circle, fill=black] (e4) at (3,3){};
\node[left] at (3,3)  {$\ss e_4$};

\node[inner sep=1.0pt, circle, fill=black] (e5) at (5,3){};
\node[right] at (5,3)  {$\ss e_5$};
\draw[-> , blue] (e4) to  (w);
\draw[-> , blue] (e5) to  (w);
\node[inner sep=1.0pt, circle, fill=black] (e42) at (2.6,2){};
\node[inner sep=1.0pt, circle, fill=black] (e46) at (3.4,2){};

\draw[-> , blue] (e42) to   (e4);

\draw[-> , blue] (e46) to  (e4);
\node[inner sep=1.0pt, circle, fill=black] (e52) at (4.6,2){};
\node[inner sep=1.0pt, circle, fill=black] (e56) at (5.4,2){};
\draw[-> , blue] (e52) to   (e5);
\draw[-> , blue] (e56) to   (e5);
\node[below] at (2.6,2)  {$\ss e_4e_2$};
\node[below] at (2.6,2) {$\vdots$};
\node[below] at (3.4,2)  {$\ss e_4e_6$};
\node[below] at (3.4,2) {$\vdots$};
\node[below] at (4.6,2) {$\ss e_5e_2$};
\node[below] at (4.6,2) {$\vdots$};
\node[below] at (5.4,2) {$\ss e_5e_6$};
\node[below] at (5.4,2) {$\vdots$};

\end{tikzpicture}
\]
Note that the level $n$ of $T_E$ has $|E^n|$ vertices.

 \end{example}

Recall that a partial isomorphism of the forest $T_E$ corresponding to a given directed graph $E$  consists of a pair $(v, w) \in E^0 \times E^0$ and a bijection $g : vE^* \to wE^*$ such that
\begin{itemize}

\item $g|_{vE^k} : vE^k \to wE^k$ is bijective for all $k\ge 1$.

\item $g(\xi e)\in g(\xi)E^1$ for $\xi\in vE^*$  and $e\in E^1$ with $r(e)=s(\xi)$.
 
 \end{itemize}
 Here $vE^k$ denotes the set of paths $\xi\in E^k$ with $r(\xi)=v$ and similarly for $vE^*$.
The set of partial isomorphisms of $T_E$ with the usual operations forms a discrete groupoid PIso$(T_E)$ with unit space $E^0$. The identity morphisms are  $id_v : vE^* \to vE^*$, the inverse of $g : vE^* \to wE^*$ is $g^{-1} : wE^* \to vE^*$, and the  multiplication is composition. We often identify $v\in E^0$ with $id_v\in$ PIso$(T_E)$. 

In Example \ref{ex}, the sum of the entries in each row of the graph adjacency matrix 
\[ A_E=\left(\begin{array}{ccc}1&1&0\\1&0&1\\0&2&0\end{array}\right)\] 
is the same, so the set of edges ending at each vertex has the same cardinality and the corresponding trees are isomorphic. In general it could happen that for some vertices $v,w$ there is no bijection $g : vE^* \to wE^*$ as above. In particular, PIso$(T_E)$ could be  a group bundle, see Example \ref{gb}. If $E$ has a single vertex, then PIso$(T_E)=$Aut$(T_E)$ is a discrete group.

\begin{example}\label{gb}
Let $E$ be the graph
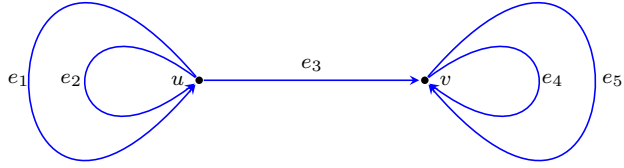
\vspace{-20mm} 
\[
\begin{tikzpicture}[shorten >=0.4pt,>=stealth, semithick]
\renewcommand{\ss}{\scriptstyle}
\node[inner sep=1.0pt, circle, fill=black]  (u) at (-1.5,0) {};
\node[left] at (u.west)  {$\ss u$};
\node[inner sep=1.0pt, circle, fill=black]  (v) at (1.5,0) {};
\node[right] at (v.east)  {$\ss v$};
\draw[->, blue] (u) to (v);
\node at (0,0.2){$\ss e_3$};
\draw[->, blue] (u) .. controls (-4.5,3.5) and (-4.5, -3.5) .. (u);
\node at (-3.9,0) {$\ss e_1$};
\draw[->, blue] (u) .. controls (-3.5,1.5) and (-3.5, -1.5) .. (u);
\node at (-3.2,0) {$\ss e_2$};
\draw[->, blue] (v) .. controls (3.5,1.5) and (3.5, -1.5) .. (v);
\node at (3.2,0) {$\ss e_4$};
\draw[->, blue] (v) .. controls (4.5,3.5) and (4.5, -3.5) .. (v);
\node at (4,0) {$\ss e_5$};
\end{tikzpicture}
\vspace{-20mm}
\]
with $E^0=\{u,v\}, E^1=\{e_1, e_2, e_3, e_4, e_5\}$ and forest $T_E$
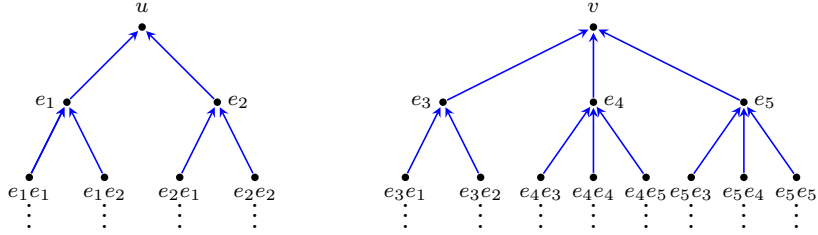
\[
\begin{tikzpicture}[shorten >=0.4pt,>=stealth, semithick]
\renewcommand{\ss}{\scriptstyle}
\node[inner sep=1.0pt, circle, fill=black]  (u) at (-3,4) {};
\node[above] at (u.north)  {$\ss u$};
\node[inner sep=1.0pt, circle, fill=black]  (v) at (3,4) {};
\node[above] at (v.north)  {$\ss v$};

\node[inner sep=1.0pt, circle, fill=black] (e1) at (-4,3){};
\node[left] at (-4,3)  {$\ss e_1$};
\node[inner sep=1.0pt, circle, fill=black] (e2) at (-2,3){};
\node[right] at (-2,3)  {$\ss e_2$};
\draw[->, blue] (e1) to  (u);
\draw[-> , blue] (e2) to  (u);
\node[inner sep=1.0pt, circle, fill=black] (e11) at (-4.5,2){};
\node[inner sep=1.0pt, circle, fill=black] (e12) at (-3.5,2){};
\node[below] at (-4.5,2)  {$\ss e_1e_1$};
\node[below] at (-4.5,2) {$\vdots$};
\draw[-> , blue] (e11) to   (e1);
\node[below] at (-3.5,2)  {$\ss e_1e_2$};
\node[below] at (-3.5,2) {$\vdots$};
\draw[-> , blue] (e11) to  (e1);
\draw[-> , blue] (e12) to  (e1);
\node[inner sep=1.0pt, circle, fill=black] (e21) at (-2.5,2){};
\node[inner sep=1.0pt, circle, fill=black] (e22) at (-1.5,2){};
\node[below] at (-1.5,2) {$\ss e_2e_2$};
\node[below] at (-1.5,2) {$\vdots$};
\node[below] at (-2.5,2) {$\ss e_2e_1$};
\node[below] at (-2.5,2) {$\vdots$};
\draw[-> , blue] (e21) to  (e2);
\draw[-> , blue] (e22) to  (e2);

\node[inner sep=1.0pt, circle, fill=black] (e3) at (1,3){};
\node[left] at (1,3)  {$\ss e_3$};

\node[inner sep=1.0pt, circle, fill=black] (e4) at (3,3){};
\node[right] at (3,3)  {$\ss e_4$};
\node[inner sep=1.0pt, circle, fill=black] (e5) at (5,3){};
\node[right] at (5,3)  {$\ss e_5$};
\draw[-> , blue] (e3) to  (v);
\draw[-> , blue] (e4) to  (v);
\draw[-> , blue] (e5) to  (v);
\node[inner sep=1.0pt, circle, fill=black] (e31) at (0.5,2){};
\node[below] at (0.5,2)  {$\ss e_3e_1$};
\node[below] at (0.5,2) {$\vdots$};
\node[inner sep=1.0pt, circle, fill=black] (e32) at (1.5,2){};
\node[below] at (1.5,2)  {$\ss e_3e_2$};
\node[below] at (1.5,2) {$\vdots$};
\draw[-> , blue] (e31) to   (e3);

\draw[-> , blue] (e32) to  (e3);
\node[inner sep=1.0pt, circle, fill=black] (e43) at (2.3,2){};
\node[below] at (2.3,2)  {$\ss e_4e_3$};
\node[below] at (2.3,2) {$\vdots$};
\node[inner sep=1.0pt, circle, fill=black] (e44) at (3,2){};
\node[below] at (3,2)  {$\ss e_4e_4$};
\node[below] at (3,2) {$\vdots$};
\node[inner sep=1.0pt, circle, fill=black] (e45) at (3.7,2){};
\node[below] at (3.7,2)  {$\ss e_4e_5$};
\node[below] at (3.7,2) {$\vdots$};
\draw[-> , blue] (e43) to   (e4);
\draw[-> , blue] (e44) to   (e4);
\draw[-> , blue] (e45) to   (e4);

\node[inner sep=1.0pt, circle, fill=black] (e53) at (4.3,2){};
\node[below] at (4.3,2)  {$\ss e_5e_3$};
\node[below] at (4.3,2) {$\vdots$};
\node[inner sep=1.0pt, circle, fill=black] (e54) at (5,2){};
\node[below] at (5,2)  {$\ss e_5e_4$};
\node[below] at (5,2) {$\vdots$};
\node[inner sep=1.0pt, circle, fill=black] (e55) at (5.7,2){};
\node[below] at (5.7,2)  {$\ss e_5e_5$};
\node[below] at (5.7,2) {$\vdots$};
\draw[-> , blue] (e53) to   (e5);
\draw[-> , blue] (e54) to   (e5);
\draw[-> , blue] (e55) to   (e5);

\end{tikzpicture}
\]
Obviously, the trees $uE^*$ and $vE^*$ are not isomorphic.
\end{example}

We are interested in self-similar actions of groupoids on the path space of directed graphs $E$ as introduced and studied in \cite{LRRW}.

\begin{dfn}\label{ss} Let $E$ be a finite directed graph with no sources, and let $G$ be a discrete groupoid with unit space $G^{(0)}=E^0$. A {\em self-similar action} $(G,E)$ of  $G$  on the path space of $E$ is given by a faithful homomorphism $G\to$ PIso$(T_E)$ such that for every $g\in G$ and $e\in d(g)E^1$ there exists a unique $h\in G$ denoted also by $g|_e$ and called the restriction of $g$ to $e$ such that
\[g\cdot(e\xi)=(g\cdot e)(h\cdot \xi)\;\;\text{for all}\;\; \xi\in s(e)E^*.\]

\end{dfn}
\begin{rmk}
We have \[d(g|_e)=s(e),\; t(g|_e)=s(g\cdot e)=g|_e\cdot s(e),\;  r(g\cdot e)=g\cdot r(e).\] This can be visualized as

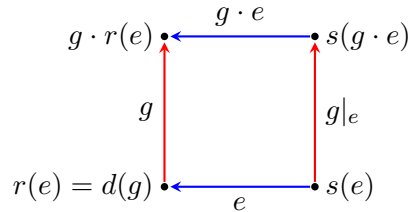
\[\begin{tikzpicture}[shorten >=0.4pt,>=stealth, thick]
\node[inner sep=1.0pt, circle, fill=black]  (u) at (-1,-1) {};
\node[inner sep=1.0pt, circle, fill=black]  (v) at (1,-1) {};
\node[inner sep=1.0pt, circle, fill=black]  (w) at (-1,1) {};
\node[inner sep=1.0pt, circle, fill=black]  (x) at (1,1) {};
\draw[-> , blue] (v) to   (u);
\draw[-> , red] (u) to   (w);
\draw[-> , red] (v) to   (x);
\draw[-> , blue] (x) to   (w);
\node[below] at (0,-1)  {$e$};
\node[above] at (0,1)  {$g\cdot e$};
\node[left] at (-1,-1)  {$r(e)=d(g)$};
\node[left] at (-1,1)  {$g\cdot r(e)$};
\node[left] at (-1,0)  {$g$};
\node[right] at (1,-1)  {$s(e)$};
\node[right] at (1,1)  {$s(g\cdot e)$};
\node[right] at (1,0)  {$g|_e$};
\end{tikzpicture}
\]

In particular, in general $s(g\cdot e)\neq g\cdot s(e)$, i.e. the source map is not $G$-equivariant as it is assumed for group actions in \cite{EP}. It is shown in Appendix A of \cite{LRRW} that a self-similar group action $(H,E)$ as in \cite{EP} determines a self-similar groupoid action $(H\times E^0, E)$ as in Definition \ref{ss}, where $H\times E^0$ is the action groupoid of the group $H$. But not any self-similar groupoid action comes from a self-similar group action.

It is possible that $g|_e=g$ for all $e\in d(g)E^1$, in which case \[g\cdot(e_1e_2\cdots e_n)=(g\cdot e_1)\cdots(g\cdot e_n).\]
\end{rmk}

 \begin{rmk}
 A self-similar groupoid action $(G,E)$    determines an action of $G$ on the graph $T_E$, in the sense that $G$ acts on both the vertex space $T_E^0$ and the edge space $T_E^1$ which are fibered over $G^{(0)}=E^0$ and intertwines the range and the source maps of $T_E$, see Definition 4.1 in \cite{De}.
  A self-similar groupoid action  extends to an action of $G$ on $E^\infty$ such that $g\cdot \xi=\eta$ if for all $n$ we have $g\cdot(\xi_1\cdots\xi_n)=\eta_1\cdots \eta_n$.
\end{rmk}

The faithfulness condition ensures that for each $g \in G$ and $\xi\in E^*$ with $d(g) = r(\xi)$, there is a unique element $g|_\xi\in G$ satisfying 
\[g\cdot(\xi\eta)=(g\cdot\xi)(g|_\xi\cdot \eta)\;\text{ for all}\; \eta\in s(\xi)E^*.\] 
By Proposition 3.6 of \cite{LRRW}, self-similar groupoid actions have the following properties: 

(1) $g|_{\xi\eta} = (g|_\xi)|_\eta$

(2) $id_{r(\xi)}|_{\xi} = id_{s(\xi)}$ 

(3) if $(h, g)\in G^{(2)}$, then $(h|_{g\cdot\xi},g|_\xi)\in G^{(2)}$ 
and $(hg)|_\xi = (h|_{g\cdot\xi})(g|_\xi)$ 

(4) $g^{-1}|_\xi=(g|_{g^{-1}\cdot\xi})^{-1}$

\noindent for $g, h \in G, \xi\in d(g)E^*$, and $\eta\in s(\xi)E^*$.
\begin{dfn}\label{trans}
A self-similar action $(G,E)$ is said to be level transitive if the induced action on $E^*$ is transitive on each $E^n$. It follows that the action is level transitive iff it is minimal on the infinite path space $E^\infty$. 

\end{dfn}
\begin{example}\label{forest} Let $E$ be the graph from Example \ref{ex}
\[\begin{tikzpicture}[shorten >=0.4pt,>=stealth, semithick]
\renewcommand{\ss}{\scriptstyle}
\node[inner sep=1.0pt, circle, fill=black]  (u) at (-2,0) {};
\node[below] at (u.south)  {$\ss u$};
\node[inner sep=1.0pt, circle, fill=black]  (v) at (0,0) {};
\node[below] at (v.south)  {$\ss v$};
\node[inner sep=1.0pt, circle, fill=black]  (w) at (2,0) {};
\node[below] at (w.south)  {$\ss w$};

\draw[->, blue] (u) to [out=45, in=135]  (v);
\node at (-1,0.7){$\ss e_2$};
\draw[->, blue] (v) to [out=-135, in=-45]  (u);
\node at (-1,-0.7) {$\ss e_3$};
\draw[->, blue] (v) to [out=45, in=135]  (w);
\node at (1,0.7){$\ss e_4$};
\draw[->, blue] (v) to (w);
\node at (1,0.1){$\ss e_5$};
\draw[->, blue] (w) to [out=-135, in=-45]  (v);
\node at (1,-0.7) {$\ss e_6$};

\draw[->, blue] (u) .. controls (-3.5,1.5) and (-3.5, -1.5) .. (u);
\node at (-3.35,0) {$\ss e_1$};

\end{tikzpicture}
\]
with $E^0=\{u,v,w\}, E^1=\{e_1, e_2, e_3, e_4, e_5, e_6\}$. Consider the groupoid $G$ with unit space $G^{(0)}=\{u,v,w\}$ and generators $a,b,c$ such that $d(a)=u, \; t(a)=v=d(b)=t(c), \; t(b)=w=d(c)$ as in the picture

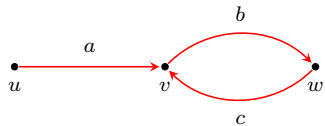
\[\begin{tikzpicture}[shorten >=0.4pt,>=stealth, semithick]
\renewcommand{\ss}{\scriptstyle}
\node[inner sep=1.0pt, circle, fill=black]  (u) at (-2,0) {};
\node[below] at (u.south)  {$\ss u$};
\node[inner sep=1.0pt, circle, fill=black]  (v) at (0,0) {};
\node[below] at (v.south)  {$\ss v$};
\node[inner sep=1.0pt, circle, fill=black]  (w) at (2,0) {};
\node[below] at (w.south)  {$\ss w$};
\draw[->, red] (u) to (v);
 \node at (-1,0.25) {$\ss a$}; 

\draw[->, red] (v) to [out=45, in=135]  (w);
\node at (1,0.7){$\ss b$};
\draw[->, red] (w) to [out=-135, in=-45]  (v);
\node at (1,-0.7) {$\ss c$};

\end{tikzpicture}
\]
Define the self-similar action $(G,E)$ given by
 \[a\cdot e_1=e_2,\;\; a|_{e_1}=u,\;\; a\cdot e_3=e_6,\;\; a|_{e_3}=b,\]
\[b\cdot e_2=e_5,\;\; b|_{e_2}=a, \;\; b\cdot e_6=e_4, \;\; b|_{e_6}=c,\] 
\[c\cdot e_4=e_2,\;\; c|_{e_4}=a^{-1},\;\; c\cdot e_5=e_6,\;\; c|_{e_5}=b.\]

Note that the action of $G$ is level transitive. It can be shown (see 
\cite{D}) that $G$ is a transitive groupoid with isotropy $\ZZ$ and therefore $C^*(G)\cong C_r^*(G)\cong M_3 \otimes C(\TT)$.

\end{example}

\begin{example} Let $E$ be the graph
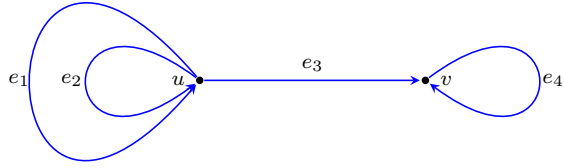
\vspace{-20mm} 
\[
\begin{tikzpicture}[shorten >=0.4pt,>=stealth, semithick]
\renewcommand{\ss}{\scriptstyle}
\node[inner sep=1.0pt, circle, fill=black]  (u) at (-1.5,0) {};
\node[left] at (u.west)  {$\ss u$};
\node[inner sep=1.0pt, circle, fill=black]  (v) at (1.5,0) {};
\node[right] at (v.east)  {$\ss v$};
\draw[->, blue] (u) to (v);
\node at (0,0.2){$\ss e_3$};
\draw[->, blue] (u) .. controls (-4.5,3.5) and (-4.5, -3.5) .. (u);
\node at (-3.9,0) {$\ss e_1$};
\draw[->, blue] (u) .. controls (-3.5,1.5) and (-3.5, -1.5) .. (u);
\node at (-3.2,0) {$\ss e_2$};
\draw[->, blue] (v) .. controls (3.5,1.5) and (3.5, -1.5) .. (v);
\node at (3.2,0) {$\ss e_4$};
\end{tikzpicture}
\vspace{-20mm}
\]
with $E^0=\{u,v\}, E^1=\{e_1, e_2, e_3, e_4\}$ and forest $T_E$
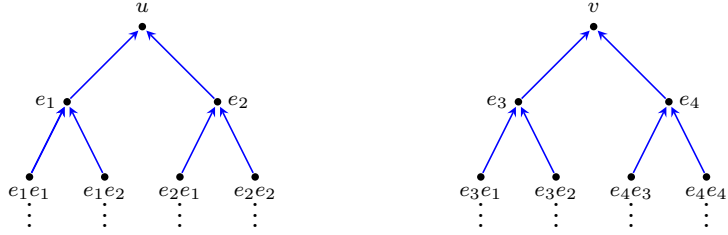
\[
\begin{tikzpicture}[shorten >=0.4pt,>=stealth, semithick]
\renewcommand{\ss}{\scriptstyle}
\node[inner sep=1.0pt, circle, fill=black]  (u) at (-3,4) {};
\node[above] at (u.north)  {$\ss u$};
\node[inner sep=1.0pt, circle, fill=black]  (v) at (3,4) {};
\node[above] at (v.north)  {$\ss v$};

\node[inner sep=1.0pt, circle, fill=black] (e1) at (-4,3){};
\node[left] at (-4,3)  {$\ss e_1$};
\node[inner sep=1.0pt, circle, fill=black] (e2) at (-2,3){};
\node[right] at (-2,3)  {$\ss e_2$};
\draw[->, blue] (e1) to  (u);
\draw[-> , blue] (e2) to  (u);
\node[inner sep=1.0pt, circle, fill=black] (e11) at (-4.5,2){};
\node[inner sep=1.0pt, circle, fill=black] (e12) at (-3.5,2){};
\node[below] at (-4.5,2)  {$\ss e_1e_1$};
\node[below] at (-4.5,2) {$\vdots$};
\draw[-> , blue] (e11) to   (e1);
\node[below] at (-3.5,2)  {$\ss e_1e_2$};
\node[below] at (-3.5,2) {$\vdots$};
\draw[-> , blue] (e11) to  (e1);
\draw[-> , blue] (e12) to  (e1);
\node[inner sep=1.0pt, circle, fill=black] (e21) at (-2.5,2){};
\node[inner sep=1.0pt, circle, fill=black] (e22) at (-1.5,2){};
\node[below] at (-1.5,2) {$\ss e_2e_2$};
\node[below] at (-1.5,2) {$\vdots$};
\node[below] at (-2.5,2) {$\ss e_2e_1$};
\node[below] at (-2.5,2) {$\vdots$};
\draw[-> , blue] (e21) to  (e2);
\draw[-> , blue] (e22) to  (e2);

\node[inner sep=1.0pt, circle, fill=black] (e3) at (2,3){};
\node[left] at (2,3)  {$\ss e_3$};

\node[inner sep=1.0pt, circle, fill=black] (e4) at (4,3){};
\node[right] at (4,3)  {$\ss e_4$};
\draw[-> , blue] (e3) to  (v);
\draw[-> , blue] (e4) to  (v);
\node[inner sep=1.0pt, circle, fill=black] (e31) at (1.5,2){};
\node[below] at (1.5,2)  {$\ss e_3e_1$};
\node[below] at (1.5,2) {$\vdots$};
\node[inner sep=1.0pt, circle, fill=black] (e32) at (2.5,2){};
\node[below] at (2.5,2)  {$\ss e_3e_2$};
\node[below] at (2.5,2) {$\vdots$};
\draw[-> , blue] (e31) to   (e3);

\draw[-> , blue] (e32) to  (e3);
\node[inner sep=1.0pt, circle, fill=black] (e43) at (3.5,2){};
\node[below] at (3.5,2)  {$\ss e_4e_3$};
\node[below] at (3.5,2) {$\vdots$};
\node[inner sep=1.0pt, circle, fill=black] (e44) at (4.5,2){};
\node[below] at (4.5,2)  {$\ss e_4e_4$};
\node[below] at (4.5,2) {$\vdots$};
\draw[-> , blue] (e43) to   (e4);
\draw[-> , blue] (e44) to   (e4);
\end{tikzpicture}
\]

Consider the groupoid $G=\la a,b,c\ra$ with $d(a)=d(b)=d(c)=u= t(a)=t(b)$ and $ t(c)=v$ as in the picture
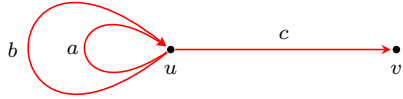
\vspace{-10mm}
\[\begin{tikzpicture}[shorten >=0.4pt,>=stealth, semithick]
\renewcommand{\ss}{\scriptstyle}
\node[inner sep=1.0pt, circle, fill=black]  (u) at (-2,0) {};
\node[below] at (u.south)  {$\ss u$};
\node[inner sep=1.0pt, circle, fill=black]  (v) at (1,0) {};
\node[below] at (v.south)  {$\ss v$};

\draw[->, red] (u) to  (v);
\node at (-0.5,0.2){$\ss c$};
\draw[->, red]  (u) .. controls (-3.5,-1) and (-3.5, 1) .. (u);
\node at (-3.3,0) {$\ss a$};
\draw[->, red] (u) .. controls (-4.5,-2) and (-4.5, 2) .. (u);
\node at (-4.1,0){$\ss b$};

\end{tikzpicture}
\vspace{-10mm}\]
Define the self-similar action $(G,E)$ by

\[a\cdot e_1=e_2,\; a|_{e_1}=b, \; a\cdot e_2=e_1, \; a|_{e_2}=a,\]
\[b\cdot e_1=e_1, \; b|_{e_1}=b, \; b\cdot e_2=e_2, \; b|_{e_2}=a,\]
\[c\cdot e_1=e_3, \;  c|_{e_1}=a, \; c\cdot e_2=e_4, \; c|_{e_2}=c.\]

Then the action is level transitive and $G_u^u=\la a,b\ra$ is isomorphic to the lamplighter group $L=(\ZZ_2)^\ZZ\rtimes \ZZ$, where the action is translation.  It can be shown that the groupoid $G$ is transitive and hence $C^*(G)\cong M_2\otimes C^*(L)$.

\end{example}
\begin{example}
Let $E$ be the graph from Example \ref{gb} with $E^0=\{u,v\}$ and $E^1=\{e_1, e_2, e_3, e_4, e_5\}$. The groupoid $G=\la a,b\ra$ where $d(a)=t(a)=u, d(b)=t(b)=v$ acts on $E^*$ by
\[a\cdot e_1=e_2,\; a|_{e_1}=u,\; a\cdot e_2=e_1,\; a|_{e_2}=a,\]
\[b\cdot e_3=e_3,\; b|_{e_3}=a,\; b\cdot e_4=e_5, \; b|_{e_4}=v,\;  b\cdot e_5=e_4,\;  b|_{e_5}=b.\]
This action is not level transitive. In fact, $G$ is a group bundle with $G_u^u\cong G_v^v\cong \mathbb Z$, so $C^*(G)\cong C(\TT)\oplus C(\TT)$.

\end{example}

\bigskip

\section{The $C^*$-algebra of the  Koopman representation}

\bigskip

In this section, we study the $C^*$-algebra $C^*(\kappa)$ associated to the Koopman representation of a level transitive self-similar action $(G,E)$ for an invariant probability measure $\nu$ on $E^\infty$. 
We assume that $|uE^1|=p\ge 2$ is constant for all $u\in E^0$, such that all trees $uE^*$  are isomorphic. In particular,  $E^\infty$ is the union of Cantor sets $uE^\infty$ for $u\in E^0$ and it has  a uniform probability measure $\nu$. 
The measure $\nu$ is obtained as the direct product of uniform  measures on each $uE^1$ such that $\ds \nu(uE^\infty)=\frac{1}{|E^0|}$. Since the action is level transitive, the measure $\nu$ is $G$-invariant.

Recall that a  trace on a  $C^*$-algebra $A$ is a linear functional $\tau:A\to \mathbb C$ such that $\tau(x^*x)\ge 0$ and $\tau(xy)=\tau(yx)$ for all $x,y\in A$. A trace is faithful if $\tau(x^*x)=0$ implies $x=0$ and is normalized  if $A$ is unital and $\tau(1)=1$.

Given an  \'etale groupoid $G$ with compact unit space, a regular Borel measure $\mu$ on $G^{(0)}$ is said to be $G$-invariant if for any bisection $U$ we have $\mu(t(U))=\mu(d(U))$. Such a regular, $G$-invariant Borel probability measure $\mu$ on $G^{(0)}$ determines a trace on $C_r^*(G)$ and on $C^*(G)$ via the formula
\[ \tau(f)=\int_{G^{(0)}}f(u)d\mu(u)\] for $f\in C_c(G)$.

For a discrete groupoid $G$, there is a  representation $\iota$ of $G$ on $\ell^2(G^{(0)})$ in which $\iota_g$ is the matrix unit $e_{t(g),d(g)}$. This is the analogue of the trivial representation of a group. The corresponding homomorphism $\pi:C^*(G)\to \Ll(\ell^2(G^{(0)}))$ takes elements of $C_c(G)$ to finite rank operators. When $G^{(0)}$ is finite, $\pi$ takes values in $M_{|G^{(0)}|}(\CC)$ and composing with the normalized trace, we obtain a tracial state on $C^*(G)$. If $G$ is not transitive, then  $\pi$ is not onto, and the normalized traces on simple summands will give other tracial states on $C^*(G)$.

Traces on $C^*(G)$ can also be  obtained from traces on the isotropy groups, see section 7 in \cite{LRRW}.

\begin{dfn}
A $C^*$-algebra $A$ is called residually finite-dimensional (RFD) if  there exist $*$-homomorphisms $\pi_n:A\to \MM_{k(n)}(\CC)$ into matrix algebras which determine a faithful map $\ds  A\to \prod_{n\ge 1} \MM_{k(n)}(\CC)$. In other words, $A$ is RFD if it has a separating  family of finite dimensional representations.
\end{dfn}
Given an embedding $\ds A\to \prod_{k\ge 1}\MM_k(\CC)$ of an RFD $C^*$-algebra $A$, if  $\tau_k$ is the standard normalized trace on $\MM_k(\CC)$, then the trace \[\tau:\prod_{k\ge 1}\MM_k(\CC)\to \CC,\;\; \tau(a_1,a_2,...)=\sum_{k=1}^\infty2^{-n}\tau_k(a_k)\]
is faithful and its restriction to $A$ yields a faithful trace. The coefficients $2^{-n}$ can be changed to possibly obtain other traces on $A$.

\begin{thm}
Let $(G,E)$ be a level transitive self-similar action such that $|uE^1|=p\ge 2$ is constant for all $u\in E^0$, and let $\nu$ be the uniform probability measure on $X=E^\infty$. Then the $C^*$-algebra $C^*(\kappa)$ of the Koopman representation of $G$ on $L^2(X,\nu)$ is RFD and it has a normalized trace $\tau_0$.
\end{thm}

\begin{proof}

The Hilbert space $\Hh=L^2(E^\infty,\nu)$ is generated by the characteristic functions $\chi_{Z(\alpha)}$  of the cylinder sets $Z(\alpha)=\{\alpha\xi:\xi\in E^\infty\}$ for $\alpha\in E^*$. Recall that the Koopman representation of $G$ is given by
\[\kappa_g:L^2(d(g)E^\infty,\nu)\to L^2(t(g)E^\infty,\nu),\; \kappa_g(f)(x)=f(g^{-1}\cdot x)\]
 and it follows that $\kappa_g(\chi_{Z(\alpha)})=\chi_{Z(g\cdot \alpha)}$.
 
 For $n\ge 0$ let $\Hh_n$ be the span of all $\chi_{Z(\alpha)}$ with $|\alpha|=n$. Then $\Hh_n$ is finite dimensional with dimension $|E^0|p^n$. We have an embedding $j_n:\Hh_n\to \Hh_{n+1}$ given by 
\[\chi_{Z(\alpha_1\cdots \alpha_n)}\mapsto \sum_{r(e)=s(\alpha_n)}\chi_{Z(\alpha_1\cdots\alpha_ne)}.\]
If $\Hh_n'$ denotes the orthogonal complement of $\Hh_{n-1}$ in $\Hh_n$ for $n\ge 1$, then $\dim \Hh_n'=|E^0|p^{n-1}(p-1)=q_n$ and \[ \Hh=L^2(E^\infty,\nu)=\Hh_0\oplus\bigoplus_{n\ge 1}\Hh_n'\] with $ \dim \Hh_0=|E^0|=q_0$. Since each of the partitions $\Pp_n$ of $E^\infty$ into cylinder sets $Z(\alpha)$ with $|\alpha|=n$ is $G$-invariant and the spaces $\Hh_n$ and $\Hh_n'$ are $G$-invariant, it follows that the Koopman representation $\kappa:G\to\Ll(L^2(X,\nu))$ is a direct sum of finite dimensional representations $\pi_n'$ and $C^*(\kappa)$ embedds into $\ds \prod_{n\ge 0}\MM_{q_n}$, so it is residually finite-dimensional. Note that $\pi_n=\kappa|_{\Hh_n}$ is isomorphic to the  representation $\sigma_n$ on $\ell^2(E^n)$ induced by the action of the groupoid on the $n$-th level $E^n$ of the forest $T_E$.

If we write $x\in C_c(G)$ as $x=(x_0,x_1,...)$ where $x_n\in \MM_{q_n}$ for $n\ge 0$ and \[ [x]_n=\left[\begin{array}{ccccc}x_0&0&\cdots &0\\0&x_1&\cdots&0\\\vdots&\vdots&\cdots&\vdots\\0&0&\cdots&x_n\end{array}\right],\] then we  define
\[\tau_0(x)=\lim_{n\to \infty}\frac{1}{|E^0|p^n}\text {Tr}[x]_n,\]
where Tr is the sum of the diagonal entries in a matrix. Obviously, $\tau_0(1)=1$.

To see that the limit exists, for $g\in G$ we denote by $[g]_n$ the matrix of order  $|E^0|p^n$ corresponding to the operator $\pi_n(g)$, in other words a partial isometry matrix with entries $0$ or $1$. If we represent $[g]_n$ as a block matrix of order $|E^0|p^{n-1}$ with diagonal blocks $g_{ii}$ of order $p$, since $[g]_n$ is obtained from the matrix $[g]_{n-1}$ by replacing the zero matrix elements with zero matrices of order $p$ and by replacing the entries equal to $1$ with the corresponding partial isometry matrix of order $p$ which has normalized trace $\le 1$, we get
\[\frac{1}{|E^0|p^n}\text{Tr}[g]_n=\frac{1}{|E^0|p^{n-1}}\sum_{i=1}^{|E^0|p^{n-1}}\frac{1}{p}\text{Tr}(g_{ii})\le \frac{1}{|E^0|p^{n-1}}\text{Tr}[g]_{n-1},\]
so the limit defining $\tau_0$ exists for $g\in G$ and for any $x\in C_c(G)$ as well.

The property $\tau_0(xy)=\tau_0(yx)$ follows from the fact that Tr$(AB)=$Tr$(BA)$. For $x\in C_c(G)$ self-adjoint we have $\|x\|_\kappa=\sup_{n\ge 0}\|x_n\|$ for $x=(x_0,x_1,...)$, where on the right hand side we have the operator norm of  $\MM_{q_n}$. We obtain
\[\tau_0(x)=\lim_{n\to \infty}\frac{1}{|E^0|p^n}\text{Tr}[x]_n=\]\[=\lim_{n\to \infty}\left[\frac{1}{p^n}\frac{\text{Tr}(x_0)}{q_0}+\frac{p-1}{p^n}\frac{\text{Tr}(x_1)}{q_1}+\frac{p-1}{p^{n-1}}\frac{\text{Tr}(x_2)}{q_2}+\cdots +\frac{p-1}{p}\frac{\text{Tr}(x_n)}{q_n}\right]\le\]\[\le \lim_{n\to \infty}\left[\left(1+\frac{1}{p}+\frac{1}{p^2}+\cdots\right)\sup_{1\le i\le n}\frac{\text{Tr}(x_i)}{q_i}\right]\le 2\sup_n\|x_n\|,\]
since Tr$(B)\le m\|B\|$ for a self-adjoint $m\times m$ matrix $B$.
We conclude that the linear functional $\tau_0$ can be extended to the whole algebra $C^*(\kappa)$, the norm closure of $C_c(G)$ in $\Ll(\Hh)$, since  any element in $C_c(G)$ is a linear combination of two self-adjoint elements.

\end{proof}

\begin{example}
Let's determine the trace $\tau_0$ on $C^*(\kappa)$, the $C^*$-algebra of the Koopman representation on $L^2(E^\infty, \nu)$ for the self-similar groupoid action $(G,E)$ in Example \ref{forest}.
Recall that  $|E^0|=3$ and $p=2$, so $\dim \Hh_n=3\cdot 2^n$. We get
\[\pi_0(a)=\left[\begin{array}{ccc}0&0&0\\1&0&0\\0&0&0\end{array}\right],\; \pi_0(b)=\left[\begin{array}{ccc}0&0&0\\0&0&0\\0&1&0\end{array}\right],\; \pi_0(c)=\left[\begin{array}{ccc}0&0&0\\0&0&1\\0&0&0\end{array}\right].\]
The $6\times 6$ matrix $\pi_1(a)$ is obtained by replacing the $21$ entry $1$ of $\pi_0(a)$ by the $2\times 2$ matrix $\ds\left[\begin{array}{cc}1&0\\0&1\end{array}\right]$ and the $0$ entries by $2\times 2$ zero matrices. The $6\times 6$ matrix  $\pi_1(b)$ replaces the $32$ entry $1$ of $\pi_0(b)$ by the $2\times 2$ matrix $\ds\left[\begin{array}{cc}0&1\\1&0\end{array}\right]$ and the $0$ entries by $2\times 2$ zero matrices. The $6\times 6$ matrix  $\pi_1(c)$ replaces the $23$ entry $1$ of $\pi_0(b)$ by the $2\times 2$ matrix $\ds\left[\begin{array}{cc}1&0\\0&1\end{array}\right]$ and the $0$ entries by $2\times 2$ zero matrices. Iterating this process, it is clear that since the diagonal entries of the $3\cdot 2^n\times 3\cdot 2^n$ matrices $\pi_n(a), \pi_n(b)$ and $\pi_n(c)$ are zero, we get
\[\tau_0(a)=\tau_0(b)=\tau_0(c)=0.\]
Note that $\ds \tau_0(u)=\tau_0(v)=\tau_0(w)=\frac{1}{3}$.
\end{example}

\bigskip

\section{Self-similar representations}

\bigskip

Consider a self-similar groupoid action $(G,E)$ such that $|uE^1|=p\ge 2$. Every $g\in G$ determines a   bijection $\sigma_g: d(g)E^1\to t(g)E^1$ given by the action $e\mapsto g\cdot e$ for $e\in d(g)E^1$,  and an element  in $G^{d(g)E^1}$  equal to the function $\varphi_g: e\mapsto g|_e$ from $d(g)E^1$ to $G$, identified with a $p$-tuple of elements in $G$. Let $\ds \mathfrak{S}(E^1)$ be the set of  bijections $uE^1\to vE^1$ for $u,v\in E^0$, which has a groupoid structure with unit space $E^0$. Note that $\mathfrak{S}(E^1)$ acts on $G^p$. The associated wreath recursion for $(G,E)$ is the groupoid homomorphism
\[\phi: G\to \mathfrak{S}(E^1)\ltimes  G^p,\;  \phi(g)=(\sigma_g, \varphi_g).\]
Indeed, we have  
\[(\sigma_g, \varphi_g)(\sigma_h,\varphi_h)=(\sigma_{gh}, \varphi_{gh})\]
for $(g,h)\in G^{(2)}$. Recall that $(gh)|_e=(g|_{h\cdot e})(h|_e)$. 
\begin{dfn}
A $p$-fold similarity of an infinite dimensional Hilbert space $\Hh$ is an isomorphism $\psi:\Hh\to \Hh^p=\underbrace{\Hh\oplus\cdots \oplus \Hh}_p$.
\end{dfn}
\begin{example}
For a level transitive self-similar groupoid action $(G,E)$ such that $|uE^1|=p\ge 2$, let $\nu$ be the invariant probability measure on $E^\infty$. The Hilbert spaces $L^2(uE^\infty, \nu)$ are decomposed into direct sums \[L^2(uE^\infty, \nu)= \bigoplus_{e\in uE^1}L^2(eE^\infty,\nu)\] and the spaces $L^2(eE^\infty,\nu)$ are naturaly isomorphic to $L^2(s(e)E^\infty,\nu)$ via
\[V_e:L^2(eE^\infty, \nu)\to L^2(s(e)E^\infty, \nu), \; V_e(f)(\xi)=\frac{1}{\sqrt{p}}f(e\xi).\]
We can view $V_e$ as partial isometries of $L^2(E^\infty,\nu)$ and then $\ds \sum_{e\in E^1}V_e$ is  a $p$-fold similarity of the Hilbert space $L^2(E^\infty, \nu)$.
\end{example}

\begin{rmk}
Recall that the graph $C^*$-algebra $C^*(E)$ is generated by partial isometries $S_e$ for $e\in E^1$ and projections $P_u$ for $u\in E^0$ such that
\[S_e^*S_e=P_{s(e)}, \;\sum_{r(e)=u}S_eS_e^*=P_u.\]
Given a representation $\pi:C^*(E)\to \Ll(\Hh)$ on a Hilbert space $\Hh$, since $\ds\sum_{u\in E^0}P_u=I$, we get a decomposition $\ds \Hh=\bigoplus_{u\in E^0}\Hh_u$ with $\Hh_u=\pi(P_u)\Hh$ and since $|uE^1|=p$ is constant, we get a $p$-fold similarity $\Hh_u\to \Hh_u^p$ given by $\xi\mapsto (\pi(S^*_{e_1})(\xi), \pi(S^*_{e_2})(\xi),...,\pi(S^*_{e_p})(\xi))$ for $uE^1=\{e_1,e_2,...,e_p\}$. Conversely, a decomposition $\ds \Hh=\bigoplus_{u\in E^0}\Hh_u$ and a $p$-fold similarity $\psi:\Hh_u\to \Hh_u^p$ for each $u\in E^0$ determines a representation of $C^*(E)$ by $\pi(S_{e_j})(\xi)=\psi^{-1}(0,...,\xi,...,0)$ with $\xi$ in position $j$ for $e_j\in uE^1$.
\end{rmk}
\begin{example}
The representation $\pi$ of $C^*(E)$ associated with the $p$-fold similarity of $\ds L^2(E^\infty,\nu)=\bigoplus _{u\in E^0}L^2(uE^\infty, \nu)$ is generated by the partial isometries $\pi(S_e)=T_e$ given by
\[T_e(f)(\xi)=\begin{cases}\sqrt{p}f(\xi')&\text{if}\;\xi=e\xi'\\0\;&\text{otherwise}\end{cases}\]
for $f\in L^2(E^\infty, \nu)$.
\end{example}
\begin{dfn}
Let $(G,E)$ be a self-similar groupoid action. If $\Hh$ is a Hilbert space bundle over $E^0$ such that each fiber $\Hh_u$ for $u\in E^0$ has a $p$-fold similarity and $\pi:C^*(E)\to \Ll(\Hh)$ is the associated representation of the graph algebra with generators $S_e$, then a representation $\rho$ of $G$ on $\Hh$ is called self-similar if $\rho(g)\pi(S_e)=\pi(S_{g\cdot e})\rho(g|_e)$.

\end{dfn}
\begin{example}
Let $E$ be a finite graph with $|uE^1|=p\ge 2$ constant for each $u\in E^0$.   If the groupoid $G$ acts on $\ds E^*=\bigcup_{u\in E^0}uE^*$ and the induced action on $E^\infty$ preserves the  measure $\nu$, then we get a representation $\rho$ of $G$ on $L^2(E^\infty, \nu)$ which is self-similar.
\end{example}

\begin{prop}
Let $(G,E)$ be a self-similar action as above, and let $\pi:C^*(E)\to \Ll(\Hh)$ be a representation associated to the similarity of a Hilbert bundle $\Hh$. A representation $\rho$ of $G$ on $\Hh$ is self-similar if and only if $\rho$ and $\pi$ generate a representation of the Cuntz-Pimsner algebra $C^*(G,E)=C^*(\Gg(G,E))$. Hence self-similar representations of $G$ are precisely restrictions onto $G$ of representations of $C^*(G,E)$.
\end{prop}
\begin{proof}

Given a self-similar groupoid action $(G,E)$, recall that the $C^*$-algebra $C^*(G,E)$    is defined as the Cuntz-Pimsner algebra of the $C^*$-correspondence \[\Mm=\Mm(G,E)=\Xx(E)\otimes_{C(E^0)}C^*(G)\] over $C^*(G)$.   Here $\Xx(E)=C(E^1)$ is the $C^*$-correspondence over $C(E^0)$ associated to the graph $E$ and $C(E^0)=C(G^{(0)})\subseteq C^*(G)$. The right action of $C^*(G)$ on $\Mm$ is the usual one and the left action is determined by the  representation
 \[W:G\to \Ll(\Mm), \;\; W_g(i_e\otimes a)=\begin{cases} i_{g\cdot e}\otimes i_{g|_e}a\;\;\text{if}\; d(g)=r(e)\\0\;\;\text{otherwise,}\end{cases}\]
 where $ i_e\in C(E^1)$ and $i_g\in C_c(G)$ are point masses for $e\in E^1, g\in G$ and  $a\in C^*(G)$. The inner product of $\Mm$ is given by
\[\la \xi\otimes a,\eta\otimes b\ra=\la\la\eta,\xi\ra\; a,b\ra=a^*\la \xi,\eta\ra\; b\]
for $\xi, \eta\in C(E^1)$ and $a,b\in C^*(G)$.

The Cuntz-Pimsner algebra $C^*(G,E)$ is generated by  $U_g, P_v$ and $S_e$, the images of $g\in G, v\in E^0=G^{(0)}$ and of $e\in E^1$ which satisfy

(1) $g\mapsto U_g$ is a representation by partial isometries of $G$ with $U_v=P_v$ for $v\in E^0$;

(2)  $S_e$ are partial isometries with $S_e^*S_e=P_{s(e)}$ and $\ds \sum_{r(e)=v}S_eS_e^*=P_v$;

(3)  $U_gS_e=\begin{cases}S_{g\cdot e}U_{g|_e}\;\mbox{if}\; d(g)=r(e)\\0,\;\mbox{otherwise;}\end{cases}$

(4) $ U_gP_v=\begin{cases}P_{g\cdot v}U_g\;\mbox{if}\; d(g)=v\\0,\;\mbox{otherwise.}\;\end{cases}$

A representation $\pi:C^*(E)\to \Ll(\Hh)$ and a representation $\rho$ of $G$ on $\Hh$ will extend to a representation of the Cuntz-Pimsner algebra $C^*(G,E)=C^*(\Gg(G,E))$ if and only if $\rho(g)\pi(S_e)=\pi(S_{g\cdot e})\rho(g|_e)$.

\end{proof}

\begin{rmk}

 Let $(G,E)$ be a self-similar action with $|uE^1|=p$. If $\rho:G\to \Ll(\Hh)$ is a self-similar representation of  $G$ on a Hilbert bundle, then with respect to the decomposition $\psi:\Hh\to \Hh^p$, each $\rho(g)$ has a $p\times p$ matrix $\rho(g)=(A_{yx})$ for $x\in d(g)E^1$ and $y\in t(g)E^1$, where
\[A_{yx}=\begin{cases}\rho(g|_x)&\text{if}\; g\cdot x=y\\0&\text{otherwise.}\end{cases}\]
We have a homomorphism  $\phi:C_c(G)\to \MM_p(C_c(G))$ of the  algebra $C_c(G)$ called matrix recursion, which is the linear extension of 
$\phi(g)=(A_{yx})$ for $x\in d(g)E^1$ and $ y\in t(g)E^1$, where
\[A_{yx}=\begin{cases}g|_x&\text{if}\; g\cdot x=y\\0&\text{otherwise.}\end{cases}\]

\end{rmk}

\begin{example}
For the self-similar action $(G,E)$ of the groupoid $G=\la a,b,c\ra$ in Example \ref{forest} we have \[\sigma_a:\{e_1, e_3\}\to \{e_2,e_6\},\; \sigma_b:\{e_2,e_6\}\to \{e_4, e_5\}, \; \sigma_c:\{e_4,e_5\}\to \{e_2,e_6\},\]
\[\varphi_a=(u,b),\; \varphi_b=(a,c),\; \varphi_c=(a^{-1},b)\] and $\phi:C_c(G)\to M_2(C_c(G))$ is given by
\[\phi(a)=\left[\begin{array}{cc}u&0\\0&b\end{array}\right],\; \phi(b)=\left[\begin{array}{cc}0&a\\c&0\end{array}\right],\; \phi(c)=\left[\begin{array}{cc}a^{-1}&0\\0&b\end{array}\right].\]
\end{example}

\end{document}